\documentclass{article}

\newcommand\nc{\newcommand}

\nc{\od}[2]{\frac{d#1}{d#2}}
\nc{\be}{\begin{equation}}
\nc{\ee}{\end{equation}}
\nc{\bd}{\begin{displaymath}}
\nc{\ed}{\end{displaymath}}
\nc{\bq}{\begin{eqnarray}}
\nc{\eq}{\end{eqnarray}}
\nc{\p}{\partial}
\nc{\ra}{\rightarrow}
\nc{\R}{\mathbb{R}}
\nc{\dx}{\mathrm{d}x}
\nc{\curl}{\operatorname{curl}}
\nc{\medotimes}{\resizebox{11pt}{!}{$\otimes$}}
\nc{\T}{\mathrm{T}}

\nc{\comment}[1]{ \texttt{\color{red} #1 } }

\usepackage{color}
\usepackage{graphics}
\usepackage{amssymb, amsmath, amsthm}

\usepackage[utf8]{inputenc}

\newtheorem{proposition}{Proposition}[section]
\newtheorem{lemma}[proposition]{Lemma}
\newtheorem{theorem}[proposition]{Theorem}
\theoremstyle{remark}
\newtheorem{remark}[proposition]{Remark}
\theoremstyle{definition}

\begin{document}


\title{Optimal energy scaling for a shear experiment in single-crystal plasticity with cross-hardening\\(Running head: Optimal energy scaling)}
\author{Keith Anguige\thanks{e-mail: keith.anguige@durham.ac.uk}\\~\\ and \\~\\
Patrick Dondl\thanks{e-mail: patrick.dondl@durham.ac.uk},\\~\\ Department of Mathematical Sciences,\\ Durham University,\\ Science Laboratories,\\ South Rd.,\\ DURHAM DH1 3LE,\\ UNITED KINGDOM} 

\date{October 14th, 2013}
\maketitle

\begin{abstract}
Consideration is given to a non-convex variational model for a shear experiment in the framework of single-crystal linearised plasticity with infinite cross-hardening. The rectangular shear sample is clamped at each end, and is subjected to a prescribed horizontal or diagonal shear, modelled by an appropriate hard Dirichlet condition. We ask: how much energy is required to impose such a shear, and how does it depend on the aspect ratio? Assuming that just two slip systems are active, we show that there is a critical aspect ratio, above which the energy is strictly positive, and below which it is zero. Furthermore, in the respective regimes determined by the aspect ratio, we prove energy scaling bounds, expressed in terms of the amount of prescribed shear.
\end{abstract}

\textbf{Keywords:} Single-crystal plasticity, cross-hardening, energy scaling, aspect ratio.

\section{Introduction and the main model}
In this article, we study optimal energy scaling and microstructure formation in the elasto-plastic deformation of single crystals, whereby, as in \cite{conti_dolz, conti_dolz2}, we adopt the framework proposed by Ortiz and Repetto \cite{Ortiz_99b}, such that the incremental displacements of inelastic solids are determined by minimising an appropriate energy functional. 

The main model considered here is a geometrically linear, continuum-mechanical elasto-plastic energy based on the assumption of strong cross-hardening, i.e. that only one slip system can be active at any given point of the material. The basic motivation for our investigation is the following question: can one devise an experiment to determine whether this single-slip condition, taken together with a surface energy which penalises geometrically necessary dislocations (GNDs), is a relevant constraint which needs to be factored into macroscopic models?

The phenomenon of cross- (or latent-) hardening ~\cite{Wu_91a, Kocks_64a, Franciosi_82a} has the effect that activity in one slip system suppresses activity in all other slip systems, which leads to a loss of convexity in the minimization problem associated with the plastic deformation of single crystals~\cite{Ortiz_99b}. In fatigue experiments, as well as experiments involving only a single-pass deformation, a lamination behaviour is frequently observed~\cite{Rasmussen_80a, Jin_84a, Bay_92a, Dondl_09a}, and this is believed to stem from cross-hardening~\cite{Ortiz_99b, Dondl_09a}. Hence, similarly as in~\cite{Dondl_09a}, the interest in including cross-hardening in our model, as a side condition on the strain effected by plastic deformation.

The surface energy we include corresponds to the core energy of GNDs. In the geometrically nonlinear setting, the density of the core energy would be of the form $|(\det F^p)^{-1} (\curl F^p) (F^p)^\T|$ (see \cite{MueMie, Sve02, Cermelli} for a discussion), where $F^p$ is the plastic deformation strain, but since we are working in the simpler geometrically linear setting, this leaves us with just the expression $|\curl F^p|$ instead.

Our main results can be described heuristically as follows. We show that the inclusion of both cross-hardening and surface energy significantly affects the energetic scaling in a particular simple shear experiment for single crystals. The experiment modelled is that of a single crystal with a defined orientation of slip systems---realisable, for example, in a B2 crystal structure---and boundary conditions such that the cuboid crystal is rigidly fixed on its (square) top and bottom faces, according to a prescribed horizontal shear, while remaining free on its four vertical sides. Without either cross hardening or surface energy, the energy infimum in our system is zero (or equal to a certain amount of plastic work if a strain-hardening energy is included) when the aspect ratio of the crystal (base-square side-length/height) is below one, while the energy increases quadratically with the imposed shear magnitude when the aspect ratio is above one. This behaviour changes dramatically when multiple slip is 
forbidden and GNDs penalised: now the energy only vanishes for aspect ratios smaller than one half, while, again, for aspect ratios above one, the energy grows quadratically with the strain imposed by the boundary conditions. For intermediate aspect ratios, a new regime of linearly growing energy arises. Thus, we conclude that the experiment proposed here can be used to discriminate between those models with surface energy and cross hardening, and those without, and hence to determine whether the inclusion of these effects in macroscopic models for single-pass plastic deformation is physically reasonable.

Kinematically, we restrict ourselves to the case of a cubic crystal structure with $\left<0 1 1\right>\{1 0 0\}$ slip systems. This means that plastic deformation occurs only on planes with normal parallel to one of the three cube axes, and in the direction of one of two Burgers vectors lying diagonally in these planes. Furthermore, we include a surface-energy term consisting of the distributional curl of the plastic deformation, as well as a constant hardening. The domain is a cuboid given by $\Omega_L = (0,1)\times (0,L) \times (0,1) \subset \R^3$, and we consider the problem of minimizing the energy
\be
\label{eq:main_energy}
E_L(u,\beta) = \int_{\Omega_L} | (\nabla u - \beta)_\mathrm{sym}|^2~dx + 
\sigma  \int_{\Omega_L} |\curl \beta| + \tau \int_{\Omega_L} |\beta|~dx
\ee
among vector-valued displacements $u \colon \Omega_L \to \R^3$ and matrix-valued plastic deformations $\beta \colon \Omega_L \to \R^{3\times 3}$, for non-negative coefficients $\sigma$ and $\tau$. Here, the subscript `$\mathrm{sym}$' denotes taking the symmetric part of the matrix in parenthesis. The first term in~\eqref{eq:main_energy} is the linearised elastic energy of the specimen, and the second term penalises GNDs, which are quantified by the row-wise curl of the Nye tensor. The third term represents a constant hardening, which for the most part will be neglected by taking $\tau = 0$ (see, however, remark after Theorem \ref{3d_thm}).

The admissible displacements for the minimisation problem are all functions $u$  
 such that $u(\cdot,0,\cdot)=(0,0,0)$, $u(\cdot,L,\cdot)= \gamma(1,0,0)$ for some parameter $\gamma\ge0$\footnote{For comparison, we also consider the diagonal-shear case, where the upper end of the specimen is fixed according to $u(\cdot,L,\cdot)= \gamma(1,1,0)$.}, and such that $(u,\beta)$ has finite energy for some $\beta$. This models an experiment in which a cuboid specimen with clamped boundary conditions on the top and bottom undergoes a simple shear deformation. Note that the free boundary conditions on the four sides of the specimen are essential in order to render this a physically realisable experiment, and that they also account for most of the technical difficulties in the analysis. While not subject to any boundary
conditions, the plastic deformation $\beta$ must be such that the absolute value of its distributional row-wise curl should be the density of a finite measure. No
further regularity condition on the test functions is required to obtain upper and lower bounds for the energy, which will be the focus of attention in the sequel.

The single-slip condition translates into a side condition on $\beta$, which, in our proposed experiment, reads as follows.   The single-crystal is oriented such that one of the cubic crystal axes points in the $x_3$ direction of the specimen, and the other two lie in the $x_1-x_2$ plane, rotated by $45^{\circ}$ with respect to the $x_1-x_2$ co-ordinate axes. Thus,
\begin{align}
\label{eq:side_cond_unrelaxed}
\beta(x) \in \Bigg\{ &s(x) \begin{pmatrix} 1/\sqrt{2} \\\ -1/\sqrt{2} \\ 1 \end{pmatrix}
\medotimes
\begin{pmatrix} 1 \\ 1\\ 0 \end{pmatrix}, 
s(x) \begin{pmatrix} 1/\sqrt{2} \\\ -1/\sqrt{2} \\ -1 \end{pmatrix}
\medotimes
\begin{pmatrix} 1 \\ 1\\ 0 \end{pmatrix}, \\
&s(x) \begin{pmatrix} 1/\sqrt{2} \\\ 1/\sqrt{2} \\ 1 \end{pmatrix}
\medotimes
\begin{pmatrix} 1 \\ -1\\ 0 \end{pmatrix}, 
s(x) \begin{pmatrix} 1/\sqrt{2} \\\ 1/\sqrt{2} \\ -1 \end{pmatrix}
\medotimes
\begin{pmatrix} 1 \\ -1\\ 0 \end{pmatrix},\nonumber \\
&s(x) \begin{pmatrix} 1 \\ 1 \\ 0 \end{pmatrix}
\medotimes
\begin{pmatrix} 0 \\ 0 \\ 1 \end{pmatrix}, 
s(x) \begin{pmatrix} 1 \\ -1 \\ 0 \end{pmatrix}
\medotimes
\begin{pmatrix} 0\\ 0\\ 1 \end{pmatrix}  \Bigg\}, \nonumber
\end{align}
for almost every $x=(x_1,x_2,x_3)\in \Omega_L$ and some coefficient $s \colon \Omega_L \to \R$. Note that there are indeed materials which exhibit such (primary) slip systems (or rotated equivalents\footnote{There may also exist further Burgers vectors lying within the respective slip planes---we show in section~\ref{sec:3d} that such considerations are irrelevant to our analysis.}), in particular materials with B2 (caesium-chloride) structure, such as the intermetallic compounds Yttrium-Zinc~\cite{Cao} or Nickel-Aluminium.

The article is organised as follows. In the next section we introduce a two-dimensional (sliced) version of our problem, and prove the key energy inequalities. Then, in section~\ref{sec:3d}, we show that the relaxed three-dimensional problem can be reduced to the two-dimensional one, simply by adding up slices. In section~\ref{sec:scalar}, we show that a further reduction to a two-dimensional scalar model of plasticity does not, in our set-up with free lateral boundary conditions, yield reasonable energy bounds. Finally, some conclusions and open problems are discussed in section~\ref{sec:conclusions}.

\section{The two-dimensional model} \label{sec:2d}

 We now introduce a 2-d model which basically corresponds to a vertical slice of the full 3-d model. Once the dimension reduction has been made, we proceed to obtain the desired upper and lower energy bounds.

\subsection{Description of the 2-d model}

Consider a rectangular specimen which occupies the domain $\overline{\Omega}_L=(0,1)\times(0,L)\subset\mathbb{R}^2$, and look for a vector displacement $u\colon \overline{\Omega}_L\rightarrow\mathbb{R}^2$ and a plastic-distortion tensor $\beta\colon \overline{\Omega}_L\rightarrow\mathbb{R}^{2\times 2}$ which minimise the linearised-plasticity functional
\be
\overline{E}_L(u,\beta) = \int_{\overline{\Omega}_L}|(\nabla u - \beta)_{\mathrm{sym}}|^2 dx + \sigma\int_{\overline{\Omega}_L}|\nabla\times\beta|,\label{E_L}
\ee
subject to the single-slip side condition
\be
\beta(x)\propto\left(
\begin{array}{rr}
1 & -1\\
1 & -1
\end{array}
\right)
\quad\textrm{or}\quad
\left(
\begin{array}{rr}
1 & 1\\
-1 & -1
\end{array}
\right),\label{side}
\ee
for a.e. $x=(x_1,x_2)\in\overline{\Omega}_L$.

The {\em curl} term appearing in (\ref{E_L}) models the energy of GNDs (with $\sigma$ interpreted as the line energy), and is defined by analogy with the measure-theoretic quantity introduced in \cite{conti_05}. This ensures that our {\em curl} term is continuous with respect to a mollification of $\beta$ which respects the side condition (\ref{side}), at least for the kind of piecewise-constant example $\beta$ given by Conti and Ortiz \cite{conti_05}. The most convenient way of explicitly making the definition is in terms of co-ordinates $(\xi,\eta)$ which are rotated at $45^{\circ}$ with respect to the $x_i$, and hence aligned with the slip directions. Thus, for $\beta\in\textrm{BV}$, say,
\be
\int_{\overline{\Omega}_L}|\nabla\times\beta| = \textrm{sup}\int_{\overline{\Omega}_L}\phi_{,\xi}\beta_{\xi\eta}+\psi_{,\eta}\beta_{\eta\xi}~d\xi d\eta,\label{curl_def}
\ee
where the {\em supremum} is taken over all $\phi, \psi\in C_0^1(\overline{\Omega}_L,[-1,1])$, and we have used the comma notation for partial derivatives.

Note that, in such rotated co-ordinates, the side condition (\ref{side}) takes the simpler form
\be
\beta(\xi,\eta)\propto\left(
\begin{array}{rr}
0 & 0\\
1 & 0
\end{array}
\right)
\quad\textrm{or}\quad
\left(
\begin{array}{rr}
0 & 1\\
0 & 0
\end{array}
\right),\label{varside}
\ee
for a.e. $(\xi,\eta)\in\overline{\Omega}_L$ , which will be extremely useful in the sequel. Also note that, for smooth $\beta$, (\ref{curl_def}) is just the $L^1$-norm of the ordinary row-wise curl of $\beta$.

Our main goal in what follows will be to obtain lower bounds on $\overline{E}_L$ as the reciprocal aspect ratio, $L$, is varied, subject to two sets of boundary conditions, corresponding to diagonal and horizontal shear, respectively. The {\em infimum} of $\overline{E}_L$ will be taken over $(u,\beta)$ which satisfy the boundary conditions under consideration, and we will write $J_L= \inf \overline{E}_L(\cdot,\cdot)$.

The boundary conditions to be treated are, when expressed in $x_i$-co-ordinates,
\begin{description}
\item[(BC1)] $u=\gamma(1,1)$ at $x_2=L$ and $u=(0,0)$ at $x_2=0$, 
\item[(BC2)] $u=\gamma(1,0)$ at $x_2=L$ and $u=(0,0)$ at $x_2=0$,
\end{description}
where $\gamma>0$ is a measure of the fixed average (diagonal or horizontal) shear.

We begin with the following pair of elementary results:

\begin{proposition}
Under \textbf{BC1} or \textbf{BC2 }, $J_L$ is a monotonically decreasing function of $L$. \label{monot}
\end{proposition}

\begin{proof}
Fix $L>0$, and choose $0<\delta<L$. Next note that we can write $\overline{\Omega}_{L-\delta}=\left(\frac{L-\delta}{L}\right)\overline{\Omega}_L\sqcup A$, where $\textrm{meas}(A)>0$. Also note that $\overline{E}_L$ is invariant under an overall scaling of $\overline{\Omega}_L$, via $\hat{u}(\lambda x)=u(x), \hat{\beta}(\lambda x)=\frac{1}{\lambda}\beta(x)$, for a scale-factor $\lambda$, say. The result follows by the super-additivity of the energy as a set function. 
\end{proof}

\begin{proposition}
Under \textbf{BC1} or \textbf{BC2 }, $J_L$ is a left-continuous function of $L$. \label{left-cont}
\end{proposition}

\begin{proof}
Fix $\epsilon, L>0$, and take a smooth pair $(u,\beta)$ on $\overline{\Omega}_L$ such that $\overline{E}_L(u,\beta)\leq J_L+\epsilon$. Now scale $u$ and $\beta$ as in the above proposition to get $(\hat{u},\hat{\beta})$ defined on $S=\left(\frac{L-\delta}{L}\right)\overline{\Omega}_L\subset\overline{\Omega}_{L-\delta}$. Extend $(\hat{u},\hat{\beta})$ to the whole of $\overline{\Omega}_{L-\delta}$ by constancy in the $x_1$-direction, thus preserving the boundary conditions at $x_2=0$ and $x_2=L-\delta$. This construction induces no extra curl on $x_1=1-\frac{\delta}{L}$, and the energy in $\overline{\Omega}_{L-\delta}\setminus S$ tends to zero as $\delta\ra 0$. Hence, $\lim_{\delta\ra 0^+}J_{L-\delta}\leq J_L + \epsilon$. Since $\epsilon$ was arbitrary, we get the desired result.
\end{proof}

\subsection{Diagonal shear}

While a diagonal shear would be rather difficult to impose in a real experiment, due to slippage, it turns out that most of the test-function constructions required for horizontal shear have simpler analogues in the \textbf{BC1} case which are much easier to explain in the first instance. For this reason, we now present the \textbf{BC1} model by way of an introduction to the more complex \textbf{BC2} problem.

In terms of rotated co-ordinates, $(\xi,\eta)$, aligned as in Figure \ref{fig1}, the Dirichlet condition of \textbf{BC1} at the top boundary becomes $(u_{\xi},u_{\eta})=\gamma (0,\sqrt{2})$, while, of course, $(u_{\xi},u_{\eta})=(0,0)$ at the bottom. Using this, it is easy to construct zero-energy test functions whenever $L>1$.

\begin{proposition}
Subject to \textbf{BC1}, $L>1\Rightarrow J_L=0$.\label{thm1}
\end{proposition}

\begin{proof}
Take a simple-shear band $S$, of width $\epsilon$, which just misses the horizontal boundaries, as in Figure \ref{fig1}. Let $u_{\xi}=u_{\eta}=\beta=0$ below $S$, $u_{\xi}=0, u_{\eta}=\sqrt{2}\gamma, \beta=0$ above $S$, and interpolate linearly between the two to ensure that
\be
\nabla u=\beta=\left(
\begin{array}{cc}
0 & 0\\
\frac{\sqrt{2}\gamma}{\epsilon} & 0
\end{array}
\right),
\ee
on $S$, in $(\xi,\eta)$ co-ordinates. Clearly the boundary conditions are satisfied by the resulting test functions, and it is trivial to check that $\overline{E}_{L}(u,\beta)=0$.
\end{proof}

\begin{figure}
\centering
\resizebox{4in}{3.7in}{\includegraphics{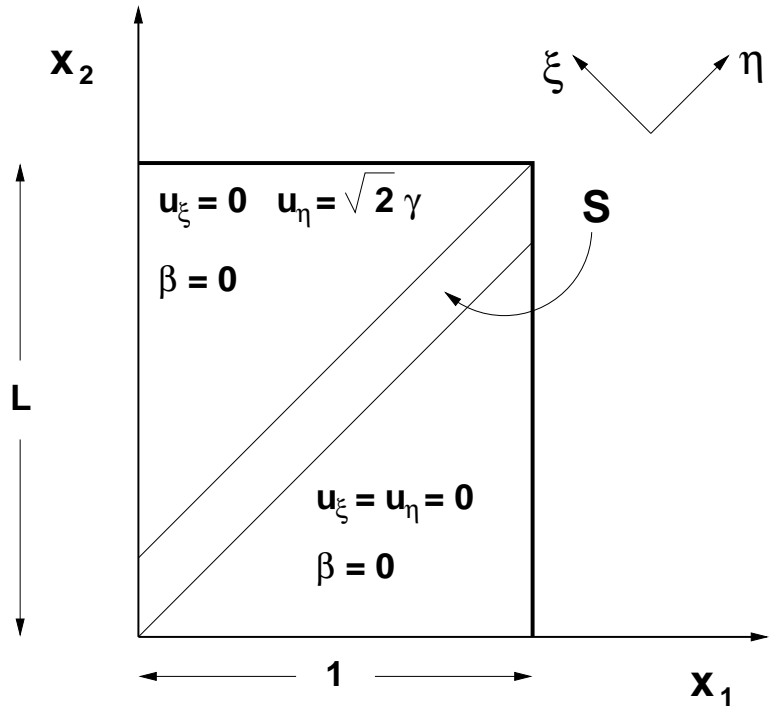}}
\caption{A minimising pair for \textbf{BC1} in the case $L>1$.}\label{fig1}
\end{figure}

\begin{proposition}
Subject to \textbf{BC1}, $L<1\Rightarrow ~ \exists~c_L>0$ such that 
\be
\frac{\gamma^2}{L}(1-L)\leq J_L\leq \min\left\{\frac{3\gamma^2}{2L}, \frac{\gamma^2}{L}(1-L) + c_L\sigma\gamma\right\}.
\ee
\label{thm2}
\end{proposition}

\begin{proof}
For the lower bound, fix $L<1$, and choose a pair $(u,\beta)$ satisfying \textbf{BC1}. By (\ref{E_L}) and (\ref{varside}), we immediately get
\be
\left\|\frac{\p u_{\eta}}{\p\eta}\right\|^2_{L^2(\overline{\Omega}_L)} +\quad\left\|\frac{\p u_{\xi}}{\p\xi}\right\|^2_{L^2(\overline{\Omega}_L)}\leq \overline{E}_L(u,\beta).\label{diag_est}
\ee 

Referring to Figure \ref{fig2}, we fix $\xi$ such that the dotted line $l_{\xi}$ (of constant $\xi$) lies in the region $R\subset\overline{\Omega}_L$ which is bounded by the two solid diagonals. Now integrate $\frac{\p u_{\eta}}{\p\eta}$ along $l_{\xi}$ from $u_{\eta}=0$ to $u_{\eta}=\sqrt{2}\gamma$ to get

\begin{figure}
\centering
\resizebox{4in}{2.8in}{\includegraphics{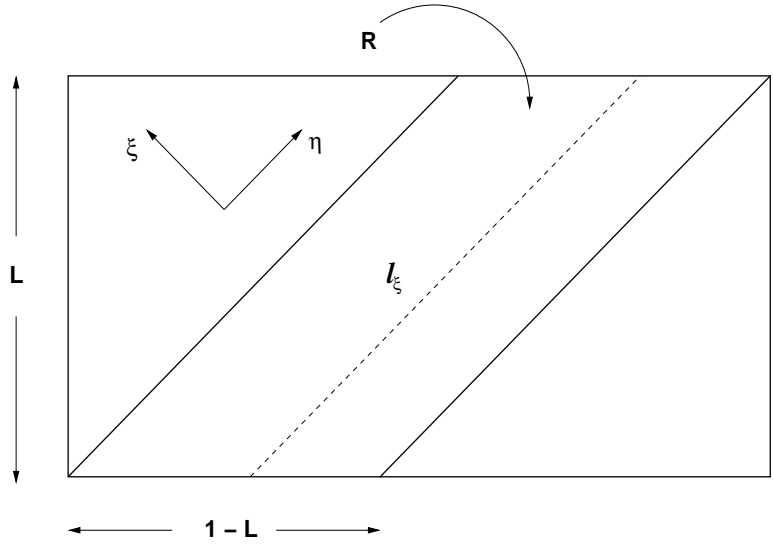}}
\caption{The line $l_{\xi}$ and the region $R$ used in the lower-bound proof of Proposition \ref{thm2}.}\label{fig2}
\end{figure}

\be
\sqrt{2}\gamma = \int_{l_{\xi}}\frac{\p u_{\eta}}{\p\eta}~d\eta,
\ee
and integrate over all such $l_\xi\subset R$, to arrive at

\be
\gamma (1-L) = \iint_R \frac{\p u_{\eta}}{\p\eta}~d\eta d\xi.
\ee

Using $|R|=L(1-L)$, along with the usual $L^2\hookrightarrow L^1$ embedding, we get

\be
\gamma(1-L)\leq\left\|\frac{\p u_{\eta}}{\p\eta}\right\|_{L^1(R)}\leq (L(1-L))^{\frac{1}{2}}\left\|\frac{\p u_{\eta}}{\p\eta}\right\|_{L^2(R)}.
\ee

Hence, by (\ref{diag_est}), $\overline{E}_L(u,\beta)\geq \gamma^2(1-L)/L$, as required.

For the upper bound, first note that the first member of the upper-bound set is attained by the minimising (linear) purely elastic deformation. Next, we get arbitrarily close to the second member with the following construction, which is depicted in Figure \ref{fig13}.

Divide $\overline{\Omega}_L$ into five regions: $R_{\epsilon}$ (a slightly narrower version of $R$ from Figure \ref{fig2}), two transition layers of width $\epsilon$ centred on the corner-diagonals of constant $\xi$, and two triangles which surround the $R_{\epsilon}$/transition-layer sandwich. 

\begin{figure}
\centering
\resizebox{4in}{2.8in}{\includegraphics{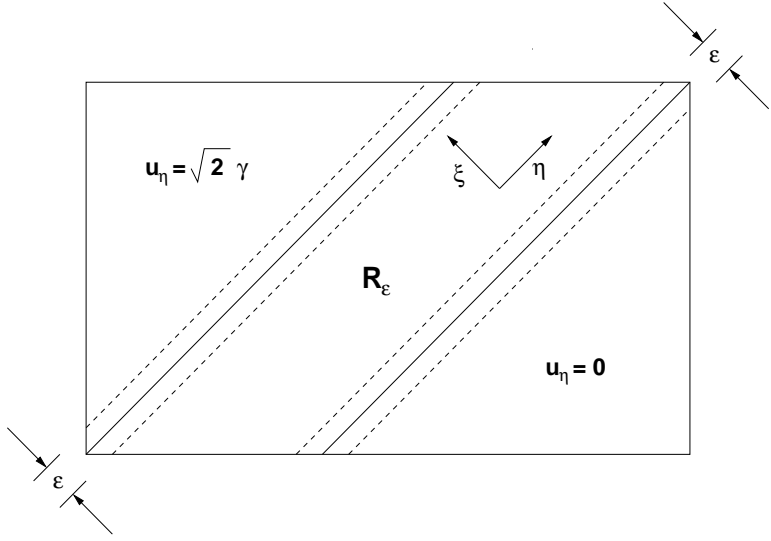}}
\caption{Construction used in the upper-bound proof of Proposition \ref{thm2}.}\label{fig13}
\end{figure}

Set $u_{\xi}=0$ on $\overline{\Omega}_L$, and, on $R_{\epsilon}$, determine $u_{\eta}$, for each fixed $\xi$, by linearly interpolating between the Dirichlet conditions at $x_2=0$ and $x_2=L$, and set $\beta_{\eta\xi}=\frac{\p u_{\eta}}{\p\xi}$. On the upper and lower triangles take $\beta=0$ and, respectively, $u_{\eta}=\sqrt{2}\gamma$ and $u_{\eta}=0$. This gives a total elastic energy away from the transition layers amounting to $\frac{\gamma^2}{L}(1-L-\sqrt{2}\epsilon)$, and zero plastic energy. 

On the transition layers, $(u,\beta)$ can be chosen such that the elastic energy is small, and the plastic energy is $\mathcal{O}(\sigma\gamma)$. To be specific, divide the upper transition layer into a small triangle, $A$, in the bottom left corner of $\overline{\Omega}_L$, and a narrow strip, $B$ (see Figure \ref{fig14}a). Then, by analogy with the proof of Proposition \ref{thm3} (see below), we can make the energy in $A$ arbitrarily small, as follows. Thus, take a function $u_{\eta}^{\beta}(\xi)$ with (piecewise) $\xi^{\alpha}$-shape, representing a steep transition from $0$ to $\frac{\gamma}{\sqrt{2}}$ along the diagonal boundary of $A$. Then set $\beta=\nabla u_{\eta}^{\beta}$, put $v=u_{\eta}-u_{\eta}^{\beta}$ and minimise $\|\nabla v\|_{L^2(A)}$, subject to $u_{\eta}=0$ on $x_2=0$ and $u_{\eta}=u_{\eta}^{\beta}$ on the diagonal boundary of $A$. This can be done by reflecting $A$ about its vertical boundary and solving the Laplace equation for $v$ on the resulting domain, which gives arbitrarily small 
energy as $\alpha\ra 0$.

Meanwhile, on $B$, interpolate $u_{\eta}$ between the $u_{\eta}^{\beta}$-profile at the boundary with $A$ and the Dirichlet condition at $x_2=L$, for each fixed $\xi$, and accommodate the resulting twist by setting $\beta_{\eta\xi}=\frac{\p u_{\eta}}{\p\xi}$.

The result of this is that we have, for small enough $\alpha(\epsilon)$, $\mathcal{O}(\epsilon)$ elastic energy and $\mathcal{O}(\sigma\gamma)$ plastic energy 
on the upper transition layer. Clearly, an analogous construction works on the lower transition layer, and hence, letting $\epsilon\ra 0$, the desired upper bound follows.

\begin{figure}
\centering
\resizebox{4in}{1.9in}{\includegraphics{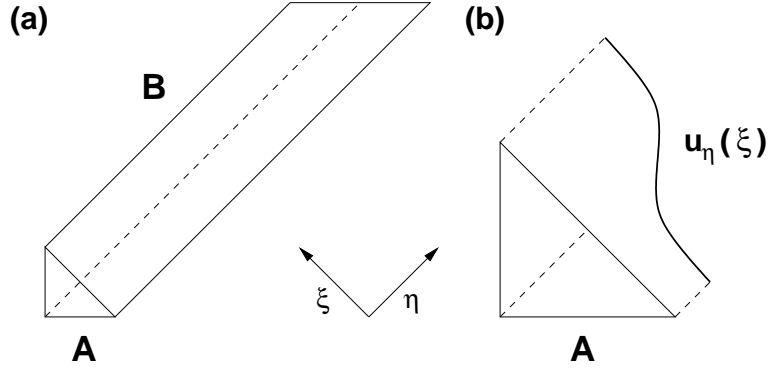}}
\caption{Construction used in the upper-bound proof of Proposition \ref{thm2}. \textbf{(a)} The upper transition layer, comprising regions $A$ and $B$, \textbf{(b)} Profile of $u_{\eta}$ on the diagonal boundary of $A$.}\label{fig14}
\end{figure}

\end{proof}

An interesting question is whether $J_{L}$ jumps from zero to a strictly positive number as $L$ decreases through unity: we now answer this question in the negative.

\begin{proposition}
Subject to \textbf{BC1}, $J_1=0$, and hence, by Proposition \ref{left-cont}, $J_{L}$ is continuous at $L=1$.\label{thm3} 
\end{proposition}

To prove Proposition \ref{thm3}, we need the aid of the following Lemma.
\begin{lemma}
The family of functions $f_{\alpha}(x) = 1-x^{\alpha},~\alpha\in(0,1]$, lies in $H^{\frac{1}{2}}(0,1)$, and $\|f_{\alpha}\|_{H^{\frac{1}{2}}(0,1)}\searrow 0$ as $\alpha\ra 0^+$. In particular, $H^{\frac{1}{2}}(0,1)$ is not embedded in $L^{\infty}(0,1)$, since $\|f_{\alpha}\|_{L^{\infty}(0,1)}=1$ for $\alpha\in(0,1]$. Moreover, the reflected functions $\bar{f}_{\alpha}(x) = 1 - (\textrm{sgn}(x)x)^{\alpha}$ lie in $H^{\frac{1}{2}}(-1,1)$, and also $\|\bar{f}_{\alpha}\|_{H^{\frac{1}{2}}(-1,1)}\searrow 0$ as $\alpha\ra 0+$.\label{lemma1}
\end{lemma}

\begin{proof}
Clearly, $f_{\alpha}\in L^2$, $\|f_{\alpha}\|_{L^2}\searrow 0$ and $\|f_{\alpha}\|_{L^{\infty}}=1$. Now, turning to the $H^{\frac{1}{2}}$ semi-norm, we define
\be
Q(\alpha)=\left|f_{\alpha}\right|_{H^{\frac{1}{2}}} = \int_0^1\int_0^1\frac{(x^{\alpha}-y^{\alpha})^2}{(x-y)^2}~dxdy,\qquad P(\alpha,x,y) = (x^{\alpha}-y^{\alpha})^2,
\ee
for $(\alpha, x, y)\in (0,1)^3$.
Thus, $Q(0)=0, Q(1)=1$, and $P(0,\cdot,\cdot)=0$. Moreover,

\be
\frac{\p P}{\p\alpha} = 2(x^{\alpha}-y^{\alpha})(x^{\alpha}\ln{x}-y^{\alpha}\ln{y})\geq 0,
\ee
so that $P(\cdot, x, y)$ is increasing, and clearly $P(\alpha, x, y)\ra P(0, x, y)=0$ as $\alpha\ra 0^+$. Thus, by dominated convergence, $Q(\alpha)\ra 0$ as $\alpha\ra 0^+$, as required. 

For $\bar{f}_{\alpha}$, the semi-norm integral can be split into two `diagonal' terms, which can be handled as above, and two `cross' terms, given by
\be
\int_{-1}^0\int_0^1\frac{(x^{\alpha}-(\textrm{sgn}(y)y)^{\alpha})^2}{(x-y)^2}~dxdy = \int_0^1\int_0^1\frac{(x^{\alpha}-z^{\alpha})^2}{(x + z)^2}~dxdz,
\ee
where we changed variables via $z=-y$ in the final equality. Clearly, the last integral is smaller than $Q(\alpha)$, and increasing w.r.t. $\alpha$, which gives the required result.
\end{proof}

\begin{proof}[Proof of Proposition \ref{thm3}]
Here it will be convenient to align the rotated $(\xi,\eta)$ co-ordinates as in Figure \ref{fig3}, with the origin taken to be at the centre of $\overline{\Omega}_1$.

\begin{figure}
\centering
\resizebox{3in}{3in}{\includegraphics{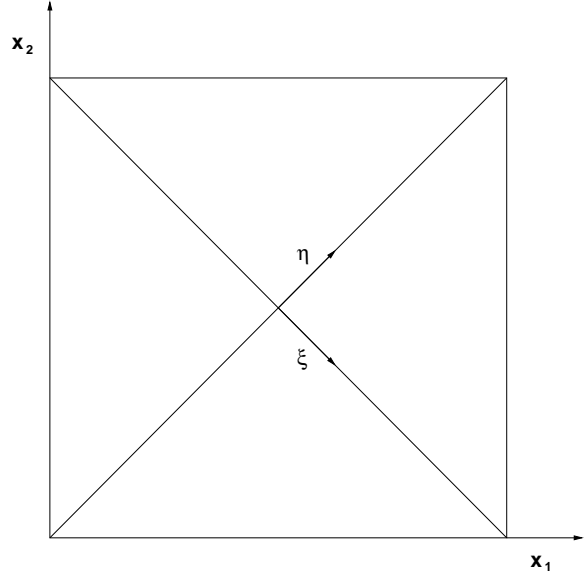}}
\caption{Orientation of the $(\xi,\eta)$ co-ordinates on $\overline{\Omega}_1$ used in the proof of Proposition \ref{thm3}.}\label{fig3}
\end{figure}

Given this, consider the sigmoid function $g_{\alpha}(\xi)$ defined by
\be
g_{\alpha}(\xi) = \frac{\gamma}{\sqrt{2}}\left(1-\textrm{sgn}(\xi)\left(\textrm{sgn}(\xi)\sqrt{2}\xi\right)^{\alpha}\right)
\ee
on the interval $(-1/\sqrt{2},1/\sqrt{2})$, which represents a steep transition from $0$ to $\sqrt{2}\gamma$ (see Figure \ref{fig4}).

\begin{figure}
\centering
\resizebox{3in}{2.5in}{\includegraphics{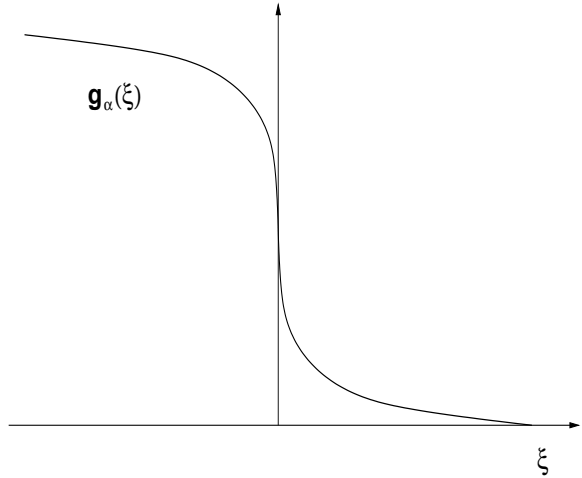}}
\caption{The function $g_{\alpha}(\xi)$ used in the proof of Proposition \ref{thm3}.}\label{fig4}
\end{figure}

Now define, in the $(\xi,\eta)$-frame, $u^{\beta}=(0,g_{\alpha}(\xi))$, and put 
\be
\beta=\nabla u^{\beta}=\left(
\begin{array}{cc}
0 & 0\\
g^{\prime}_{\alpha}(\xi) & 0
\end{array}
\right)
\ee
Thus, $\beta$ is curl-free, and, for a general test function $u$, we define $v=u-u^{\beta}$ and set $v_{\xi}=0$ ($\Leftrightarrow u_{\xi}=0$, consistent with \textbf{BC1}). The aim is now to show that the minimum of $\|\nabla v_{\eta}\|_{L^2}$, subject to the boundary conditions
\be
v_{\eta}(\xi,\eta) = \left\{
\begin{array}{rcc}
-g_{\alpha}(\xi) & : & x_2=0\\
\sqrt{2}\gamma - g_{\alpha}(\xi) & : & x_2=1
\end{array}
\right.,
\ee 
tends to zero as $\alpha\ra 0^+$: this will be achieved by using the continuity of the (Dirichlet) solution operator to Laplace's equation, $\Delta v_{\eta}=0$.
 
In order to do this, it is convenient to consider reflected (about $x_1=0$) Dirichlet data on the double of $\overline{\Omega}_1$, with resulting lateral boundaries then identified to make a cylinder. Thus, we work on the manifold $\widetilde{\Omega}_1=\mathbb{S}^1\times (0,1)$, with boundary $\mathbb{S}^1\sqcup\mathbb{S}^1$, and the Dirichlet data for $v_{\eta}$ on each boundary $\mathbb{S}^1$ have the same shape as $\bar{f}_{\alpha}$ from Lemma \ref{lemma1} (see Figure \ref{fig5}). By the symmetry of this data, and invariance of $\Delta v_{\eta}=0$ under reflections, we see that solving the Dirichlet problem on $\widetilde{\Omega}_1$, for a function still denoted by $v_{\eta}$, gives a solution of the mixed Dirichlet/Neumann problem on $\overline{\Omega}_1$. Now, Proposition 1.7 on p.307 of \cite{taylor} tells us that the solution map to the Dirichlet problem on $\widetilde{\Omega}_1$ takes $H^{\frac{1}{2}}(\p\widetilde{\Omega}_1)$ continuously into $H^1(\widetilde{\Omega}_1)$, while our Lemma \ref{lemma1} 
gives, modulo trivial 
scalings applied to $\bar{f}_{\alpha}$, $\|\textrm{Tr}(v_{\eta})\|_{H^{\frac{1}{2}}(\p\widetilde{\Omega}_1)}\ra 0$ as $\alpha\ra 0^+$, as desired.
\end{proof}

\begin{remark}
In the above construction, if we use a piecewise-linear transition function with 3-pieces, rather than $g_{\alpha}$, then we get an $\mathcal{O}(\gamma^2)$ amount of energy, as can easily be checked, using the trace inequality. 
\end{remark}

\begin{figure}
\centering
\resizebox{3in}{3in}{\includegraphics{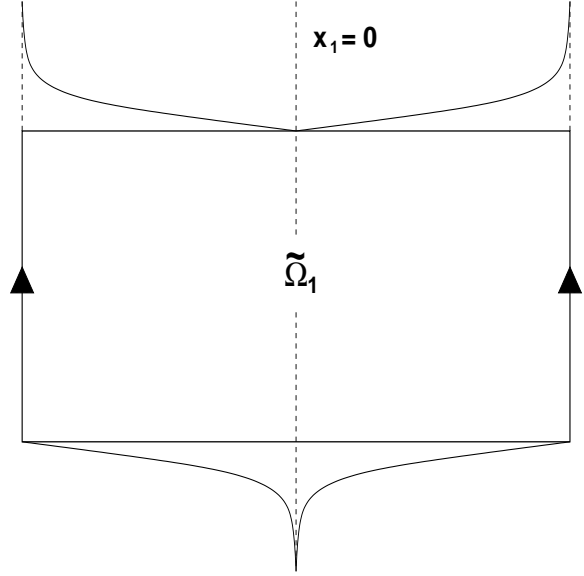}}
\caption{Shape of the Dirichlet data on $\widetilde{\Omega}_1$ used in the proof of Proposition \ref{thm3}.}\label{fig5}
\end{figure}

\subsection{Horizontal shear}

In this section we minimise $\overline{E}_L$ subject to the horizontal-shear condition \textbf{BC2}. In a $(\xi,\eta)$-frame aligned as in Figure \ref{fig1}, the Dirichlet condition at $x_2=L$ is $(u_{\xi},u_{\eta})=\gamma(-1/\sqrt{2},1/\sqrt{2})$, while $(u_{\xi},u_{\eta})=(0,0)$ at $x_2=0$.

\begin{proposition}
Subject to \textbf{BC2}, $L>2\Rightarrow J_L=0$.\label{thm4} 
\end{proposition}

\begin{proof}
We use basically the same construction as in the proof of Proposition \ref{thm1}, except that here, since $L>2$, there is enough room to fit two alternating, non-intersecting shear bands into $\overline{\Omega}_{L}$, each of which misses the Dirichlet boundary (see Figure \ref{fig6}). 

To be explicit, we take two shear bands, given by $S_1=\{x_1+x_2\leq L\}\cap\{x_1+x_2\geq 1+L/2\}$ and $S_2=\{x_2-x_1\geq 0\}\cap\{x_2-x_1\leq -1+L/2\}$, which just touch at the lateral boundary, $x_1=1$. Above $S_1$ set $u_{\xi}=-\gamma/\sqrt{2}, u_{\eta}=\gamma/\sqrt{2}, \beta=0$, between $S_1$ and $S_2$ set $u_{\xi}=0, u_{\eta}=\gamma/\sqrt{2}, \beta=0$, and below $S_2$ put $u_{\xi}=u_{\eta}=\beta=0$. Then, linearly interpolating $u$ across $S_1$ and $S_2$, and requiring that $\nabla u=\beta$ on $S_1\cup S_2$ clearly gives a pair of test functions which satisfy \textbf{BC2}, and for which $\overline{E}_L=0$.  
\end{proof}

\begin{figure}
\centering
\resizebox{2.25in}{3.15in}{\includegraphics{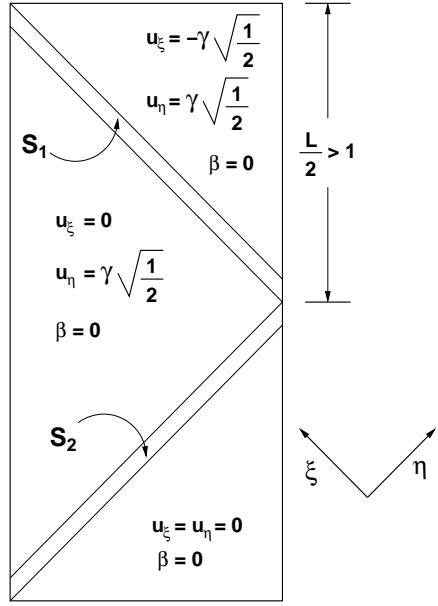}}
\caption{Zero-energy construction for Proposition \ref{thm4}.}\label{fig6}
\end{figure}

A certain amount of energy is required to produce a horizontal shear when $L$ is sufficiently small, which is the content of the following two propositions.

\begin{proposition}
Subject to \textbf{BC2}, $L<1\Rightarrow\exists~c_{L}>0$ such that 
\be
\label{horiz_bd}
\frac{\gamma^2}{2L}(1-L)\leq J_{L}\leq\min\left\{\frac{\gamma^2}{2L}, \frac{\gamma^2}{2L}(1-L) + c_L\sigma\gamma\right\}.    
\ee
\label{thm5}
\end{proposition}

\begin{proof}
The proof is very similar to that of Proposition \ref{thm2}. 

Thus, for the lower bound, take a test pair $(u,\beta)$, and integrate $\frac{\p u_{\eta}}{\p\eta}$ over the strip $R$ (see Figure \ref{fig2}), to get 
\be
\left\|\frac{\p u_{\eta}}{\p\eta}\right\|_{L^2}^2\geq\frac{\gamma^2(1-L)}{4L},
\ee
and integrate $\frac{\p u_{\xi}}{\p \xi}$ over an analogous strip pointing in the $\xi$-direction to get the same estimate for $\left\|\frac{\p u_{\xi}}{\p\xi}\right\|_{L^2}$. Hence, $\overline{E}_L(u,\beta)\geq\gamma^2(1-L)/2L$, as required.

For the upper bound, it is easy to see that the first member of the upper-bound set in (\ref{horiz_bd}) is attained by the minimising purely elastic deformation, while the second member can be approached by a construction similar to, but more complicated than, that used in the proof of Proposition \ref{thm2}, the difference being that here we need two transition layers for each of $u_{\xi}$ and $u_{\eta}$, and we have to take a little care on the regions where they intersect.

\begin{figure}
\centering
\resizebox{4.1in}{3in}{\includegraphics{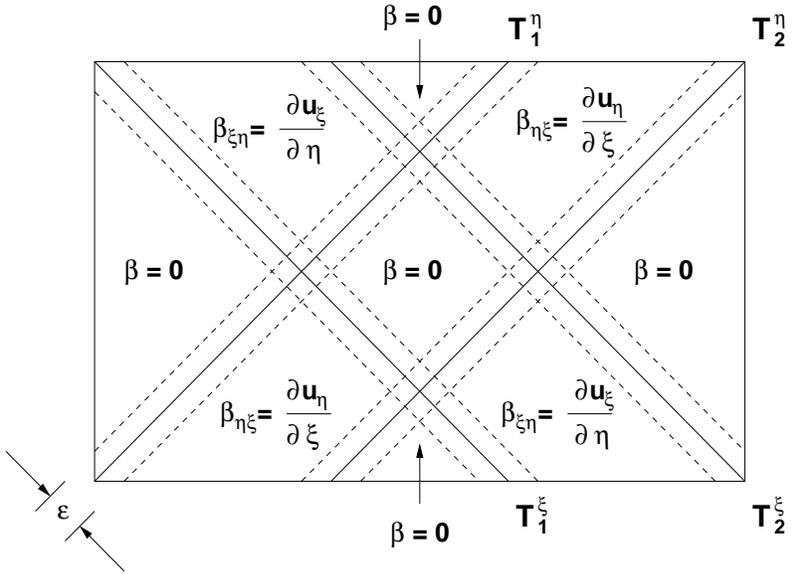}}
\caption{Transition layers used in the proof of Proposition \ref{thm5}.}\label{fig15}
\end{figure}

Thus, referring to Figure \ref{fig15}, we define, for $i=1,2$, the transition layers $T_i^{\eta}$ to be narrow strips of width $\epsilon$, centred on the corner-diagonals of constant $\xi$, and $T_i^{\xi}$ to be strips of the same width, but centred on the corner-diagonals of constant $\eta$. Between the $T_i^{\eta}$, $u_{\eta}$ is obtained by linear intepolation between the Dirichlet conditions, in the $\eta$ direction, while above and below the $T_i^{\eta}$-sandwich, $u_{\eta}$ is set equal to the appropriate Dirichlet condition. The other component of the displacement, $u_{\xi}$, is obtained on the complement of the $T_i^{\xi}$ in an analogous manner.

Next, $u_{\eta}$ is defined on $T_1^{\eta}$ as follows. Divide the part of $T_1^{\eta}$ below the intersection with $T_1^{\xi}$ into two parts, $A$ and $B'$, such that $A$ is the same as in Figure \ref{fig13} and $B'$ is a truncated version of $B$ (see Figure \ref{fig16}). Define $u_{\eta}$ on $A$ in exactly the same way as in Proposition \ref{thm2} with $\alpha>0$ small and fixed, and, for each fixed $\xi$, interpolate from the diagonal boundary of $A$ up to the unique linear profile on the upper constant-$\eta$ boundary of $B'$ which connects $u_{\eta}$ continuously across $T_1^{\eta}$. On the remainder of $T_1^{\eta}$, define $u_{\eta}$ by interpolating up to the Dirichlet condition at $x_2=L$, for each fixed $\xi$.

\begin{figure}
\centering
\resizebox{3in}{2.5in}{\includegraphics{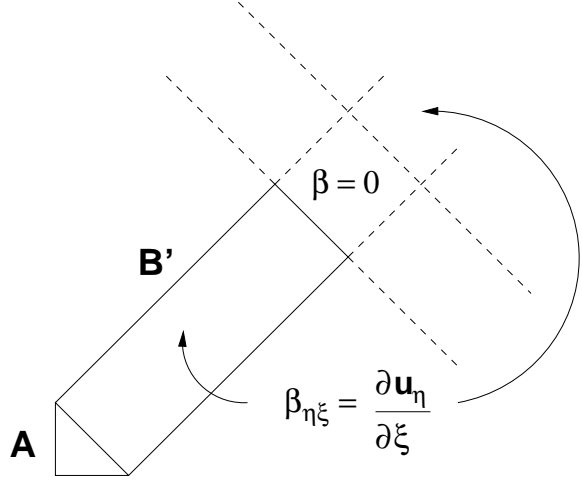}}
\caption{Subsets of $T_1^{\eta}$ used in the proof of Proposition \ref{thm5}.}\label{fig16}
\end{figure}

On $T_2^{\eta}$, $u_{\eta}$ is determined by a trivial symmetry transformation applied to $u_{\eta}$ on $T_1^{\eta}$, and likewise for $u_{\xi}$ on the $T_i^{\xi}$. For the plastic distortion, let $\beta=\nabla u^{\beta}$ on $A$, set $\beta=0$ on $T_1^{\eta}\cap T_i^{\xi}$, and on the rest of $T_1^{\eta}$ let $\beta_{\eta\xi}=\frac{\p u_{\eta}}{\p\xi}$. Define $\beta$ on the remaining transition layers by symmetry, and otherwise by appealing to Figure \ref{fig15}.

As a result of this, there is very little elastic energy on the transition layers, while the plastic energy is concentrated there (particularly on the boundaries). More precisely, we have $\|\frac{\p u_{\eta}}{\p\eta}\|_{L^2}^2=\mathcal{O}(\epsilon)$ on $T_i^{\eta}$ and $\|\frac{\p u_{\xi}}{\p\xi}\|_{L^2}^2=\mathcal{O}(\epsilon)$ on $T_i^{\xi}$, provided $\alpha(\epsilon)$ was chosen small enough, while, by the symmetry of the construction, $\frac{\p u_{\eta}}{\p\xi} + \frac{\p u_{\xi}}{\p\eta}=\mathcal{O}(1)$, pointwise, on the transition-layer intersections. Also note that, by symmetry, $\frac{\p u_{\eta}}{\p\xi} + \frac{\p u_{\xi}}{\p\eta}=0$ on the central square, where $\beta=0$, and thus there is no off-diagonal elastic energy there. 

Finally, it is easy to see that all the contributions to the plastic energy are $\mathcal{O}(\sigma\gamma)$, and hence, for our test functions,
\be
\overline{E}_L(u,\beta) = \frac{\gamma^2}{2L}(1-L) + \mathcal{O}(\sigma\gamma) + \mathcal{O}(\epsilon),
\ee
so that, letting $\epsilon\ra 0$, the desired upper bound follows. 
\end{proof}

We have to work a little harder to prove strict positivity of the energy in the intermediate case $1<L<2$.

\begin{proposition}
Subject to \textbf{BC2}, $1<L<2\Rightarrow\exists~c_L>0$ such that
\be
\frac{c_L\sigma\gamma^2}{\sigma + \sqrt{\sigma^2 + 2c_L\gamma^2}}\leq J_L\leq\textrm{min}\left\{\frac{\gamma^2}{2L},2\sqrt{2}\gamma\sigma\right\}. \label{eq:thm6}
\ee\label{thm6} 
\end{proposition}

\begin{proof}
For the lower bound, we combine a recent generalisation of Korn's inequality from \cite{gar_pons} which is nicely adapted to our energy functional with a Lemma from \cite{conti_05} which allows us to control (one component of) $\beta$ via the curl, provided the single-slip condition is satisfied.

First, let $1<L<2$, and, for test functions $(u,\beta)$ satisfying \textbf{BC2}, we define the average-symmetric tensor $A$ via
\be
A(u,\beta) = (\nabla u-\beta) - \frac{1}{\left|\overline{\Omega}_L\right|}\int_{\overline{\Omega}_L}(\nabla u - \beta)_{\textrm{skew}}~dx.
\ee

Then, by Theorem 11 of \cite{gar_pons}, there exists a constant $C(\overline{\Omega}_L)$ such that \footnote{Note that, at the time of writing, (\ref{gar_pons}) is only known to hold in 2-d, although analogous inequalities in higher dimensions have been conjectured by Neff {\em et al.}\cite{neff}.}
\be
\int_{\overline{\Omega}_L}\left|A_{\textrm{skew}}\right|^2~dx \leq C\left(\int_{\overline{\Omega}_L}\left|A_{\textrm{sym}}\right|^2~dx +\left\|    
\nabla\times A\right\|^2_{L^1(\overline{\Omega}_L)}\right),\label{gar_pons}
\ee
and hence there exists a constant, skew-symmetric matrix $K(u,\beta)$ such that
\be
\int_{\overline{\Omega}_L}\left|(\nabla u-\beta) - K(u,\beta)\right|^2 \leq C \overline{E}_{L}(u,\beta)\left(1 + \frac{\overline{E}_L(u,\beta)}{\sigma^2}\right).\label{K_est}
\ee

Now, consider the square $Q\subset\overline{\Omega}_L$, of side $\frac{1}{\sqrt{2}}$, which is bounded by the lines $\{x_1+x_2=L/2\}$, $\{x_2-x_1=-1+L/2\}$, $\{x_1+x_2=1+L/2\}$ and $\{x_2-x_1=L/2\}$ (See Figure \ref{fig7}): for convenience, we have put the origin of the $(\xi,\eta)$-co-ordinates at the bottom corner of $Q$.

\begin{figure}
\centering
\resizebox{2.5in}{3in}{\includegraphics{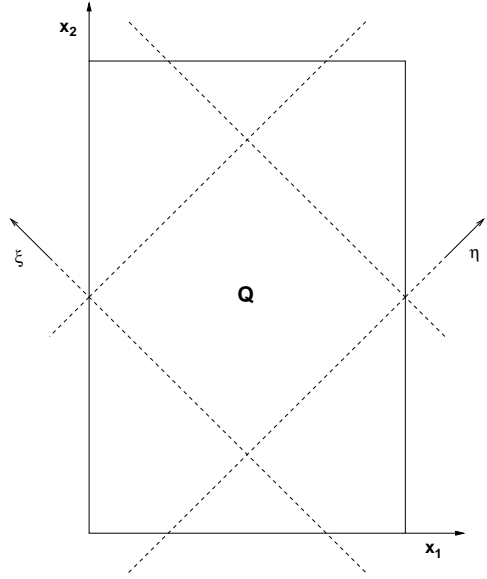}}
\caption{The region $Q$ used in the proof of Proposition \ref{thm6}.}\label{fig7}
\end{figure}

By analogy with the proof of Lemma 4.3 of \cite{conti_05}, we define the following subsets of $Q$:
\bq
\omega^{(\xi)} & = & \left\{\xi\in (0,1): \beta_{\eta\xi}(\xi,\eta)\neq 0~\textrm{for a.e.}~\eta\in(0,1)\right\}, \\
\omega^{(\eta)} & = & \left\{\eta\in (0,1): \beta_{\xi\eta}(\xi,\eta)\neq 0~\textrm{for a.e.}~\xi\in(0,1)\right\}, 
\eq
and, from the side condition (\ref{side}), we conclude that at least one of $\omega^{(\xi)}$ and $\omega^{(\eta)}$ must be a null set. For the remainder of the proof, we may assume w.l.o.g. that $\omega^{(\eta)}$ is a null set.

Let 
\be
P=\overline{\Omega}_L\cap\left\{\frac{L}{2}\leq x_1 + x_2\leq 1 + \frac{L}{2}\right\},
\ee
so that $P$ is, so to speak, the extension of $Q$ in the $\xi$-direction.

Then, since $\omega^{(\eta)}$ is null, we obtain
\be
\left\|\beta_{\xi\eta}\right\|_{L^1(P)}\leq C\left\|\p_{\xi}\beta_{\xi\eta}\right\|_{L^1(P)}\leq\frac{C}{\sigma} \overline{E}_L(u,\beta),\label{beta_est}
\ee
and it follows from (\ref{varside}), (\ref{K_est}), (\ref{beta_est}) and $L^2\hookrightarrow L^1$ embedding that,
\be
\left\|\frac{\p u_{\xi}}{\p\xi}\right\|_{L^1(P)}^2\leq C \overline{E}_L \label{uxi_1}
\ee
and
\bq
\left\|\frac{\p u_{\xi}}{\p\eta} - K_{\xi\eta}\right\|_{L^1(P)}^2 & \leq & C \overline{E}_L\left(1 + \frac{\overline{E}_L}{\sigma^2}\right) + 2\left\|\beta_{\xi\eta}\right\|_{L^1(P)}^2\\
 & \leq & C \overline{E}_L\left(1 + \frac{\overline{E}_L}{\sigma^2}\right).\label{uxi_2}
\eq

Define strips $T_1, T_2\subset P$ by

\be
T_1 = P\cap\{x_1+x_2\leq 1\}, \qquad T_2 = P \cap\{x_1+x_2\geq L\},
\ee
such that $T_1$ intersects the bottom boundary, and $T_2$ the top boundary, in a set of length $1-L/2$.

The strategy now is to show that 
\bq
\|u_{\xi}\|_{L^1(T_1)} & \leq & C\overline{E}_L^{\frac{1}{2}},\label{T1}\\
\|u_{\xi} + \gamma/\sqrt{2}\|_{L^1(T_2)} & \leq & C\overline{E}_L^{\frac{1}{2}},\label{T2}\\
K_{\xi\eta}(u,\beta) & = & \mathcal{O}(\overline{E}_L^{\frac{1}{2}}).\label{KUp}
\eq
Then, integrating $\frac{\p u_{\xi}}{\p\eta}$ between $T_1$ and $T_2$ will give a lower bound on $K_{\xi\eta}$, and a contradiction if $\overline{E}_L$ is too small relative to the average shear $\gamma$. 

We begin by integrating $\frac{\p u_{\xi}}{\p\xi}$ from $(\xi_0(\eta),\eta)$ on the $x_2=0$ boundary of $T_1$ to the point $(\xi,\eta)$ in the interior (see Figure \ref{fig10}), to get
\be
u_{\xi}(\xi,\eta) = \int_{\xi_0}^{\xi}\frac{\p u_{\xi}}{\p\xi}(\xi',\eta)~d\xi',
\ee
which, upon integration over $(\xi,\eta)\in T_1$, gives 
\bq
\iint_{T_1}|u_{\xi}|~d\xi d\eta & \leq & \iint_{T_1}\int_{\xi_0(\eta)}^{\xi}\left|\frac{\p u_{\xi}}{\p\xi}\right|(\xi',\eta)~d\xi'd\xi d\eta\\
& \leq & C\left\|\frac{\p u_{\xi}}{\p\xi}\right\|_{L_1(\overline{\Omega}_L)}\\
& \leq & C\overline{E}_L^{\frac{1}{2}},
\eq
and therefore (\ref{T1}) holds. The estimate (\ref{T2}) follows by an analogous integration on $T_2$.

\begin{figure}
\centering
\resizebox{2in}{2.2in}{\includegraphics{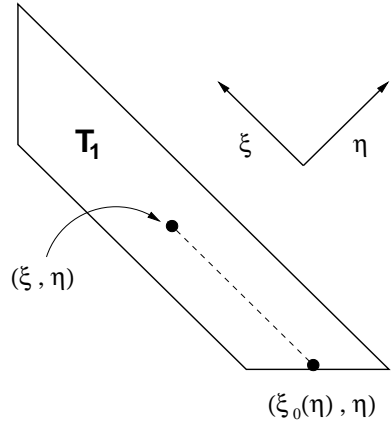}}
\caption{Construction used in the proof of Eq.(\ref{T1}).}\label{fig10}
\end{figure}

For the upper bound on $K_{\xi\eta}$, consider the small triangle $R_1 = T_1\cap\left\{x_1-x_2\geq\frac{L}{2} \right\}$ next to the $x_2=0$ boundary (Figure \ref{fig11}b).

\begin{figure}
\centering
\resizebox{4.65in}{2in}{\includegraphics{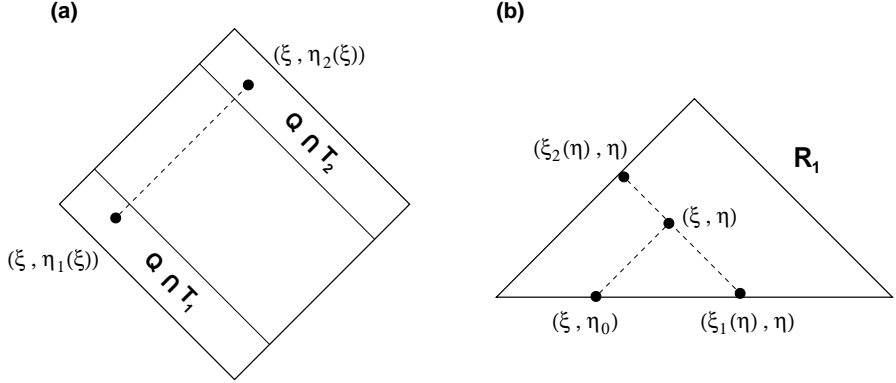}}
\caption{Constructions for the lower \textbf{(a)} and upper \textbf{(b)} bounds on $K_{\xi\eta}$ in Proposition \ref{thm6}.}\label{fig11}
\end{figure}

Integrate $\frac{\p u_{\xi}}{\p\eta}$ from $x_2=0$ to the point $(\xi,\eta)$, holding $\xi$ fixed:
\be
u_{\xi}(\xi,\eta) = \int_{\eta_0}^{\eta}\frac{\p u_{\xi}}{\p\eta}(\xi,\eta')~d\eta',\label{uxi_3}
\ee
then integrate (\ref{uxi_3}) between $\xi_1(\eta)$ and $\xi_2(\eta)$, and sum over all line segments $(\xi_1,\xi_2)$ lying in $R_1$, to get
\be
\iint_{R_1}u_{\xi}(\xi,\eta)~d\xi d\eta = \int_{\eta_1}^{\eta_2}\int_{\xi_1(\eta)}^{\xi_2(\eta)}\int_{\eta_0(\xi)}^{\eta}\frac{\p u_{\xi}}{\p\eta}(\xi,\eta')~d\eta' d\xi d\eta,\label{uxi_4}
\ee
where $[\eta_1,\eta_2]=\Pi_{\eta}(R_1)$, and $\Pi_{\eta}$ denotes projection onto the $\eta$-axis.

Adding and subtracting $K_{\xi\eta}$ to the integrand on the rhs of (\ref{uxi_4}), and using (\ref{uxi_2}) and (\ref{T1}), we arrive at
\be
|K_{\xi\eta}|\leq C\overline{E}_L^{\frac{1}{2}}\left(1 + \left(1+\frac{2\overline{E}_L}{\sigma^2}\right)^{\frac{1}{2}}\right).\label{uxi_8}
\ee

For the lower bound on $K_{\xi\eta}$, we work on the square $Q$, and, in particular, the subsets $Q\cap T_i$ thereof (see Figure \ref{fig11}b).

Firstly, we integrate $\frac{\p u_{\xi}}{\p\eta}$ between the two points $(\xi,\eta_1(\xi))\in Q\cap T_1$ and $(\xi,\eta_2(\xi))\in Q\cap T_2$:
\be
\int_{\eta_1(\xi)}^{\eta_2(\xi)}\frac{\p u_{\xi}}{\p\eta}(\xi,\eta')~d\eta' = u_{\xi}(\xi,\eta_2) - u_{\xi}(\xi,\eta_1).\label{uxi_5}
\ee 
Then, we integrate (\ref{uxi_5}) over $\xi\in\Pi_{\xi}(Q)$, $\eta_1\in\Pi_{\eta}(T_1)$ and $\eta_2\in\Pi_{\eta}(T_2)$, to arrive at
\be
\iiiint\frac{\p u_{\xi}}{\p\eta}(\xi,\eta')~d\eta' d\xi d\eta_1 d\eta_2 = \frac{1}{\sqrt{2}}\left( 1-\frac{L}{2} \right)\left(\iint_{T_2}u_{\xi}~d\xi d\eta_2 - \iint_{T_1}u_{\xi}~d\xi d\eta_1\right).\label{uxi_6}
\ee

Thus, adding and subtracting $K_{\xi\eta}$ from the integrand on the lhs of (\ref{uxi_6}), and using (\ref{uxi_2}), (\ref{T1}) and (\ref{T2}), gives us
\be
|K_{\xi\eta}|\geq C\left(\frac{\gamma}{\sqrt{2}} - \overline{E}_L^{\frac{1}{2}}\left(1 + \left(1+\frac{2\overline{E}_L}{\sigma^2}\right)^{\frac{1}{2}}\right)\right),\label{uxi_7}
\ee
which, together with (\ref{uxi_8}), implies
\be
\overline{E}_L^{\frac{1}{2}}\left(1 + \left(1+\frac{2\overline{E}_L}{\sigma^2}\right)^{\frac{1}{2}}\right)\geq C\gamma.
\ee
Squaring, and using $(a+b)^2\leq 2(a^2 + b^2)$ plus the quadratic formula, leads to the required lower bound on $J_L$.

For the upper bound on $J_L$, we note that the first member of the upper-bound set in (\ref{eq:thm6}) is attained by the minimising purely elastic deformation satisfying \textbf{BC2}, while the second member is attained by a simple crossing-shear-band construction, as follows.

Referring to Figure \ref{fig12}, inscribe, in $\overline{\Omega}_L$, crossing shear bands $S_1$ and $S_2$ which just miss the Dirichlet boundaries, and which have width $\frac{1}{\sqrt{2}}(L-1)$. Then define $u_{\xi}$ and $u_{\eta}$ by linearly interpolating between the boundary conditions across the respective $S_i$. Finally, set $\beta=\nabla u$ on $\overline{\Omega}_L\setminus (S_1\cap S_2)$ and $\beta=0$ on $S_1\cap S_2$. An easy calculation shows that the elastic energy vanishes, while the plastic energy is given by $2\sqrt{2}\gamma\sigma$, as required.
\end{proof}

\begin{figure}
\centering
\resizebox{3.4in}{1.6in}{\includegraphics{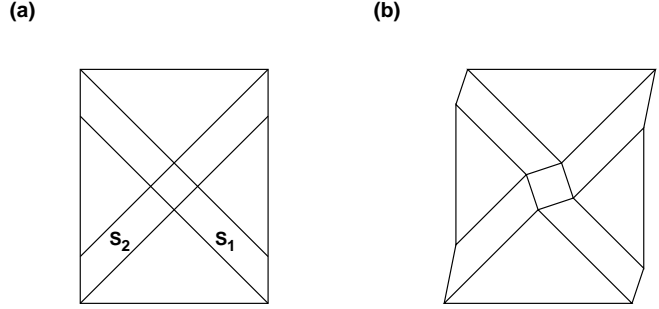}}
\caption{Upper-bound construction for Proposition \ref{thm6} \textbf{(a)} in the reference configuration, and \textbf{(b)} the deformed configuration.}\label{fig12}
\end{figure}

\begin{remark}
If the single-slip condition is dropped then $J_L=0$ for $L>1$, by essentially the same construction as in Figure \ref{fig12}, the difference being that we may take $\beta=\nabla u$ on the whole of $\overline{\Omega}_L$, and not just away from the shear-band intersection. 
\end{remark}

Finally, we show that, as in the diagonal-shear case, the transition from zero to positive energy is continuous.

\begin{proposition}
Subject to \textbf{BC2}, $L=2\Rightarrow J_L=0$.\label{thm7}
\end{proposition}

\begin{proof}
The proof is similar to that of Proposition \ref{thm3}.

Take $L=2$, and let the $(\xi,\eta)$-co-ordinates be centred at the mid-point of the right-hand lateral boundary, as shown in Figure \ref{fig8}.

\begin{figure}
\centering
\resizebox{2.1in}{3in}{\includegraphics{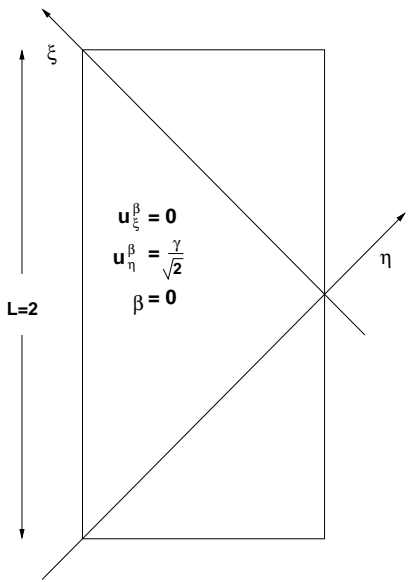}}
\caption{Orientation of $(\xi,\eta)$-co-ordinates used in the proof of Proposition \ref{thm7}.}\label{fig8}
\end{figure}

We construct test functions $(u^{\beta},\beta)$ as follows. On $\{\xi>0,\eta<0\}$ let $(u_{\xi}^{\beta},u_{\eta}^{\beta})=\gamma(0,1/\sqrt{2})$ and $\beta=0$, on $\{\xi>0,\eta>0\}$ let
\be
u_{\xi}^{\beta} = -\frac{\gamma}{\sqrt{2}}\left(\sqrt{2}\eta\right)^{\alpha},\quad u_{\eta}^{\beta}=\frac{\gamma}{\sqrt{2}},\quad \beta=\nabla u^{\beta},
\ee
and on $\{\xi<0,\eta<0\}$ let
\be
u_{\xi}^{\beta} = 0,\quad u_{\eta}^{\beta}=\frac{\gamma}{\sqrt{2}}\left(1-(-\sqrt{2}\xi)^{\alpha}\right),\quad \beta=\nabla u^{\beta},
\ee
for $\alpha\in[0,1]$.

Clearly, $\nabla\times\beta=0$, and, writing $v=u-u^{\beta}$ as usual, we wish to make $\overline{E}_L=\|(\nabla v)_{\textrm{sym}}\|_{L^2}^2$ arbitrarily small by letting $\alpha\ra 0$, subject to the boundary conditions
\be
(v_{\xi},v_{\eta})=\left\{
\begin{array}{rcl}
\left(\frac{\gamma}{\sqrt{2}}\left((\sqrt{2}\eta)^{\alpha}-1\right),0\right) & : & x_2=2,\\
\left(0,\frac{\gamma}{\sqrt{2}}\left(-1 + (-\sqrt{2}\xi)^{\alpha}\right)\right) & : & x_2=0
\end{array}
\right.. 
\ee
Since we have
\be
\left\|(\nabla v)_{\textrm{sym}}\right\|_{L^2}^2\leq C\left(\|\nabla v_{\xi} \|_{L^2}^2 + \|\nabla v_{\eta}\|_{L^2}^2\right),
\ee
it is enough to show that $\|\textrm{Tr}(v)\|_{H^{\frac{1}{2}}(\p\overline{\Omega}_2)}\ra 0$ as $\alpha\ra 0$. The proof of this follows just as in the diagonal-shear case.
\end{proof}

\section{Extension to 3-d: B2 crystals}  \label{sec:3d}

A sceptical reader might be tempted to dimiss our 2-d model as somewhat artificial, but in fact the results of section~\ref{sec:2d} carry over almost immediately to the case of crystals with $B2$ symmetry. 

As described in the introduction, two of the slip-plane normals in the B2 crystal are aligned with the $(\xi,\eta)$-axes, as before, and the third is aligned with the $\zeta$ (resp. $x_3$)-axis, which is taken to be pointing out of the page.

In co-ordinates $x_i$ aligned with the shear-sample faces, the domain occupied by the crystal is taken to be $\Omega_L=(0,1)\times(0,L)\times(0,1)$, and we impose a horizontal shear of magnitude $\gamma$, such that the boundary conditions take the form 
\be
u=\gamma(1,0,0)\quad\textrm{at}\quad x_2=L,\quad u=(0,0,0)\quad\textrm{at}\quad x_2=0.\label{bc3}
\ee

For the purposes of deriving the required lower energy bound, the single-slip side condition on the plastic distortion $\beta\colon\Omega_L\rightarrow\mathbb{R}^{3\times 3}$ is now taken to be, in the $(\xi,\eta,\zeta)$-frame,
\be
\beta(\xi,\eta,\zeta)\in\left\{
\begin{pmatrix}
0\\
s_2\\
s_3
\end{pmatrix}
\medotimes
\begin{pmatrix}
1\\
0\\
0
\end{pmatrix}
,
\begin{pmatrix}
s_1\\
0\\
s_3
\end{pmatrix}
\medotimes
\begin{pmatrix}
0\\
1\\
0
\end{pmatrix}
,
\begin{pmatrix}
s_1\\
s_2\\
0
\end{pmatrix}
\medotimes
\begin{pmatrix}
0\\
0\\
1
\end{pmatrix}
\right\},\label{3dslip}
\ee
for a.e. $(\xi,\eta,\zeta)\in\Omega_L$, and there is no restriction on the $s_i(\xi,\eta,\zeta)$. Note that this is a somewhat weaker condition than the one in~\eqref{eq:side_cond_unrelaxed}, in the sense that we allow any Burgers vector which lies within a given slip plane, and not just those in the direction of the cubic-lattice-face diagonals. Of course, any lower energy bound for the relaxed condition still holds for the original one.

In this setting, for a.e. $(\xi,\eta,\zeta)$, we have
\be
\beta = \begin{pmatrix}
0 & \beta_{\xi\eta} & \beta_{\xi\zeta}\\
\beta_{\eta\xi} & 0 & \beta_{\eta\zeta}\\
\beta_{\zeta\xi} & \beta_{\zeta\eta} & 0
\end{pmatrix},
\ee
and for a.e. slice of constant $\zeta$ (call it $S_{\zeta}$), the 2-d side condition holds: $\beta_{\xi\eta}=0$ or $\beta_{\eta\xi}=0$ for a.e. $(\xi,\eta)\in S_{\zeta}$.

Now, consider the energy functional~\eqref{eq:main_energy}, with $\tau$ set to zero and subject to the relaxed side condition~\eqref{3dslip}. Denote by $(\nabla u - \beta)_{\textrm{sym}}^{2\times 2}$ the $(\xi,\eta)$-components of the elastic strain, and, for fixed $\zeta$, define the reduced slice energy, $E_L^{\zeta}$, by 
\be
E_L^{\zeta}(u,\beta) = \iint_{S_{\zeta}}\left|\left(\nabla u -\beta\right)_{\textrm{sym}}^{2\times2}\right|^2 + \left|\p_{\xi}\beta_{\xi\eta}\right| + \left|\p_{\eta}\beta_{\eta\xi}\right|~d\xi d\eta.
\ee
Then, clearly, $E_L(u,\beta)\geq\int_0^1 E_L^{\zeta}~d\zeta$, and, since all of our 2-d lower bounds from previous sections apply to $E_L^{\zeta}$, for a.e. $\zeta$, they also go through for the $B2$ crystal oriented as above. Meanwhile, the upper bounds for the energy can be recovered in the following way. First, simply take the same test function, $(u,\beta)$, as for the 2-d case, extend constantly in the $\zeta$ direction, and let $\beta_{\zeta\cdot}=\beta_{\cdot\zeta}=0$. This establishes the upper bounds subject to the relaxed side condition~\eqref{3dslip}. The original side condition can be recovered by introducing fine oscillations between the two possible Burgers vectors, in the active slip planes along the slip-plane normal, whereby it should be noted that the curl term picks up a factor of $\sqrt{2}$ when doing the lamination.


In summary, we have proved
\begin{theorem}
Modulo the substitution $\sigma\ra\sqrt{2}\sigma$ in the upper-bound set, the upper and lower bounds of Propositions \ref{thm4}--\ref{thm7} also hold for the energy (\ref{eq:main_energy}) with $\tau=0$, subject to (\ref{bc3}), as applied to a $B2$ crystal occupying the domain $\Omega_L=(0,1)\times(0,L)\times(0,1)$ with slip-plane normals oriented as above, and such that the cross-hardening condition (\ref{3dslip}) is satisfied. \label{3d_thm}
\end{theorem}

\begin{remark}
For the case of a non-zero, constant strain hardening, i.e. $\tau>0$, the energy scaling is as follows. For $L\ge 2$, the minimal energy is now at most $\tau \gamma$, which is easy to see by inspecting the 2-d test function used previously. For $L<1$, we still have the same quadratic scaling of the energy with respect to (large) $\gamma$ as before, since the upper-bound test function of Figure \ref{fig15} gives a strain-hardening energy of order $\tau\gamma$. In the regime $1\le L <2$, the linear scaling of the energy with respect to $\gamma$ is preserved: the proof for the lower bound goes through unchanged, while inspection of the test function for the upper bound yields an additional contribution to the energy of order $\tau \gamma$. Furthermore, for a physically realistic model, we would expect $\tau$ to be small compared to both the elastic modulus (implicitly equal to unity here) and $\sigma$. This is due to the fact that the $L^1$-penalty on $\beta$ only accounts for self-hardening, i.e. the hardening 
which occurs if the specimen is deformed purely in single slip---all other hardening is taken care of by the single-slip side condition and the surface energy.
\end{remark}

\section{The scalar model} \label{sec:scalar}

In this section, we take a brief (and salutary) look at an ostensibly simplified scalar version of the model (\ref{E_L})-(\ref{curl_def}) which is analogous to that treated in section 4 of \cite{conti_05}. Rather than imposing a soft Dirichlet condition on the whole of $\p\overline{\Omega}_L$ as in \cite{conti_05}, however, we continue to use a hard Dirichlet condition to model a fixed horizontal shear imposed by clamping both ends of the single-crystal specimen.

As before, the rectangular specimen occupies the domain $\overline{\Omega}_L=(0,1)\times(0,L)\subset\mathbb{R}^2$, but this time we focus attention on just the $x_1$-component of the deformation field, denoted by $u\colon\overline{\Omega}_L\ra\mathbb{R}$, and the plastic strain, $\beta\colon\overline{\Omega}_L\ra\mathbb{R}^2$, which are to be obtained by minimising the functional

\be
E^s_L(u,\beta)=\int_{\overline{\Omega}_L}|\nabla u-\beta|^2 dx + \int_{\overline{\Omega}_L}|\nabla\times\beta|,\label{Es}
\ee
subject to the side condition $\beta_1 = \pm \beta_2$, in the sense of distributions in $\overline{\Omega}_L$, and where the second term in (\ref{Es}) is the measure-theoretic quantity defined in \cite{conti_05}. This quantity can be obtained from (\ref{curl_def}) by making the substitutions $\beta_{\xi\eta}\ra\beta_{\eta}$ and $\beta_{\eta\xi}\ra\beta_{\xi}$ in the usual rotated $(\xi,\eta)$ co-ordinates. In these co-ordinates, the side condition on $\beta$ takes the form: $\beta_{\xi}=0$ or $\beta_{\eta}=0$ for a.e. $(\xi,\eta)\in\overline{\Omega}_L$.

We prescribe free boundary conditions on the two `vertical' sides of $\overline{\Omega}_L$, and a fixed average shear, $\gamma$, in the $x_1$ direction on the remaining two sides. In other words, $\langle n,\nabla u-\beta\rangle=0$ on $\p\overline{\Omega}_L\cap\{x_1=0, 1\}$, where $n$ is the unit normal to $\p\overline{\Omega}_L$, $u = 0$ on $\p\overline{\Omega}_L \cap \{x_2 = 0\}$ and $u = \gamma $ on $\p\overline{\Omega}_L \cap \{x_2 = L\}$.

Now we write $J^s_L= \inf E^s_L(\cdot,\cdot)$, and ask how $J^s_L$ varies with the $L$. It turns out that analysing this question is no easier than in the full vector-valued case, and that, in fact, there is a qualitative difference between the results for the scalar model and those for the vector-valued model subject to either \textbf{BC1} or \textbf{BC2}.

First of all we note that $J^s_L$ is left-continuous and monotonically decreasing, by the same arguments used in Propositions \ref{monot} and \ref{left-cont}. Then the next proposition, combined with the foregoing remark, implies that $J^s_L=0$ for $L\geq\frac{1}{2}$ and that there is no jump in the energy at $L=\frac{1}{2}$: the difference between this result and those for \textbf{BC1} and \textbf{BC2} is explained by the fact that there is no constraint on the vertical deformation in the scalar model, which is rather unphysical. Whether the energy is strictly positive for $L<\frac{1}{2}$ remains an open question.

\begin{proposition}
Subject to $u=0$ at $x_2=0$ and $u=\gamma$ at $x_2=\frac{1}{2}$, we have $J^s_{\frac{1}{2}}=0$.
\label{thm8} 
\end{proposition}

\begin{proof}
Let $L=\frac{1}{2}$, and centre the $(\xi,\eta)$ co-ordinates at the mid-point of the top boundary, as shown in Figure \ref{fig9}.

\begin{figure}
\centering
\resizebox{3.5in}{2in}{\includegraphics{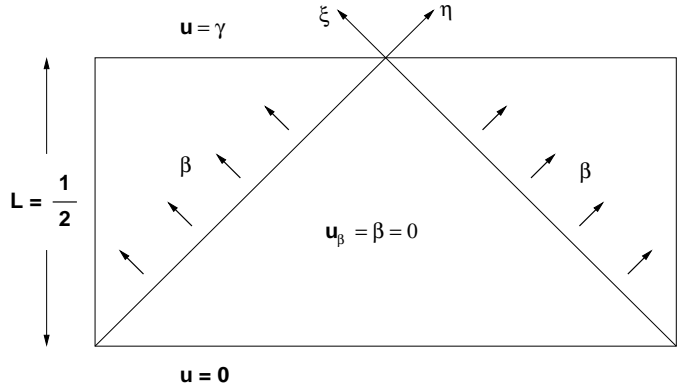}}
\caption{Test functions used in the proof of Proposition \ref{thm8}.}\label{fig9}
\end{figure}

Now, for each $\alpha\in (0,1)$, construct test functions $(u_{\beta},\beta)$ by setting $u_{\beta}=\beta=0$ on $\{\xi>0, \eta<0\}$,
\be
u_{\beta} = \gamma(2\sqrt{2}\xi)^{\alpha},\quad\beta=\nabla u_{\beta}\quad\textrm{on}\quad\{\xi>0,\eta<0\},
\ee
and
\be
u_{\beta} = \gamma(2\sqrt{2}\eta)^{\alpha},\quad\beta=\nabla u_{\beta}\quad\textrm{on}\quad\{\xi<0,\eta>0\}.
\ee

Thus, $\nabla\times\beta=0$, and, in the now familiar way, we write $v=u-u_{\beta}$, where $u$ is the unknown deformation, and we look for $v$ which minimises $\|\nabla v\|_{L^2}^2$, subject to the Dirichlet conditions
\be
v = \left\{
\begin{array}{ccl}
 0 & : & x_2=0\\
\gamma(1-(2\sqrt{2}\xi)^{\alpha}) & : & x_2=\frac{1}{2}, \eta<0\\
\gamma(1-(2\sqrt{2}\eta)^{\alpha}) & : & x_2=\frac{1}{2}, \xi<0
\end{array}
\right..
\ee

The fact that $\|\nabla v\|_{L^2}$ can be made arbitrarily small as $\alpha\ra 0$ follows by analogy with the proof of Proposition \ref{thm3}, given the shape of the Dirichlet data.
 
\end{proof}

\section{Conclusions} \label{sec:conclusions}
\begin{figure}
\centering
\resizebox{4.5in}{2.5in}{\includegraphics{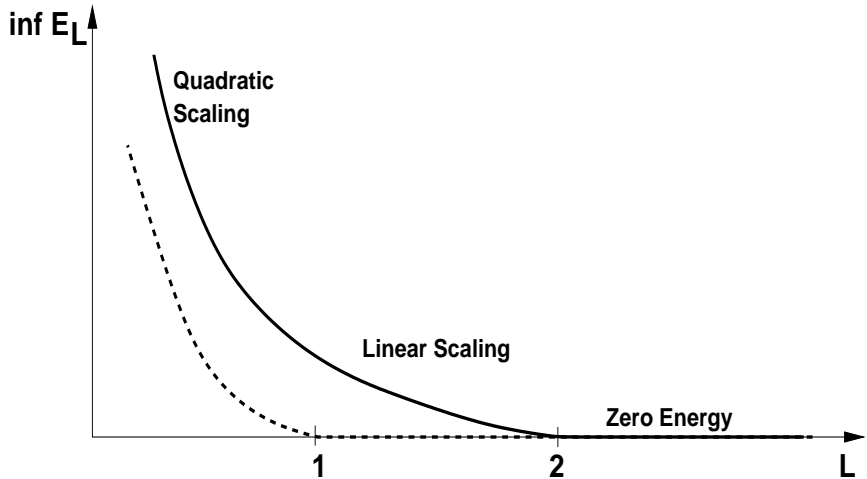}}
\caption{Schematic of the different scaling regimes for the energy~\eqref{eq:main_energy} with $\tau=0$. The solid line represents the energy when both $\sigma>0$ and the single-slip condition is enforced---in this case, the energy vanishes for $L\ge 2$, scales linearly with the (large) applied strain $\gamma$ for $1\le L < 2$ , and scales quadratically for $L<1$. The dashed line shows the energy when either $\sigma = 0$ or the single-slip side condition is dropped. In this case, the linear energy-scaling regime is no longer present.}
\label{fig_energy}
\end{figure}

We have derived energy-scaling results and explicit bounds for the elasto-plastic energy in a realistic single-crystal shear experiment. The different scaling regimes under various assumptions are summarised in Figure~\ref{fig_energy}. As discussed in the introduction, the results obtained should provide a means of discriminating between two different classes of models in crystal plasticity: those with strong cross-hardening and surface energy, and those without. It is, however, open as to whether our results carry over from the B2 case to more common crystal structures, such as face-centred-cubic or hexagonal-close-packed lattices. 

Of course, in this article, we have only considered a geometrically linear elasto-plastic energy, and while our upper bounds should still hold in the geometrically nonlinear case, using a multiplicative decomposition of the strain, obtaining the lower energy bounds is a much more challenging problem, and is the subject of ongoing work. A further objective is to carry out numerical simulations of the deformation predicted by our modelling.

Finally, the experiments themselves are also work in progress, and will be performed in collaboration with the Max Planck Institut f{\"u}r Eisenforschung in D{\"u}sseldorf---one of the main practical problems in this regard will be to impose the boundary conditions without slippage or other spurious deformation.

\section*{Acknowledgements} This work was supported by the Deutsche Forschungsgemeinschaft (DFG), within the research unit 797 `Microplast'. We would also like to thank an anonymous referee for helpful suggestions and comments.

\end{document}